\documentclass[a4paper,12pt]{article}
\usepackage{amsmath,amssymb}
\newtheorem{theorem}{Theorem}[section]
\newtheorem{lemma}[theorem]{Lemma}
\newtheorem{proposition}[theorem]{Proposition}
\newtheorem{corollary}[theorem]{Corollary}
\newenvironment{proof}{\normalsize {\sc Proof}:}{{\hfill $\Box$}}
\newenvironment{proofof}[1]{\normalsize {\sc Proof of #1}:}{{\hfill $\Box$}}
\def\margin_comment#1{\marginpar{\sffamily{\small #1\par}\normalfont}}
\def \N {\mathbb{N}}
\def \Z {\mathbb{Z}}
\def \G {\mathcal{G}}
\def \K {\mathcal{K}}
\def \C#1#2{\mathcal{C}(#1,#2)}
\def \ch {\rm chr}
\def \adj {\sim}
\def \nonadj {\not\sim}
\newenvironment{mylist}{\begin{list}{}{
\setlength{\parskip}{0mm}
\setlength{\topsep}{1mm}
\setlength{\parsep}{0mm}
\setlength{\itemsep}{0.5mm}
\setlength{\labelwidth}{7mm}
\setlength{\labelsep}{3mm}
\setlength{\itemindent}{0mm}
\setlength{\leftmargin}{12mm}
\setlength{\listparindent}{6mm}
}}{\end{list}}
\title{Generalising some results about right-angled Artin groups to graph products of groups}
\author{Derek F. Holt and Sarah Rees}
\date{14th October 2011}
\begin{document}
\maketitle

\begin{abstract}
We prove three results about the graph product
$G=\G(\Gamma;G_v, v \in V(\Gamma))$  of groups $G_v$ over a graph $\Gamma$.
The first result generalises a result of Servatius, Droms and Servatius, proved by them for right-angled Artin groups; we prove a necessary
and sufficient condition on a finite graph $\Gamma$ for the kernel of the map
from $G$ to the associated direct
product to be free (one part of this result already follows from a result
in S. Kim's Ph.D.
thesis). The second result generalises
a result of Hermiller and  \u{S}uni\'{c}, again from right-angled Artin groups;
we prove that, for a graph $\Gamma$ with finite chromatic number, $G$ has a
series in which every factor is a free product of vertex groups.
The third result provides an alternative proof of a theorem due to Meier,
which provides necessary and sufficient conditions on a finite
graph $\Gamma$ for $G$ to be hyperbolic.
\end{abstract}

\section{Introduction}
Given a simplicial graph $\Gamma$, and vertex groups $G_v$ for each
$v \in V(\Gamma)$, we use the notation $\G(\Gamma;G_v, v \in V(\Gamma))$ to
denote the associated graph product of groups. Where there is no ambiguity in
the choice of the vertex groups $G_v$, we may abbreviate this to $\G(\Gamma)$.  

The group $\G(\Gamma)$ is the `largest' group generated by the vertex groups
such that $G_u$ and $G_v$ commute whenever $u$ and $v$ are joined by an edge in
$\Gamma$. More precisely, it is the quotient of the free product of
the groups $G_v$ by the normal closure of the set of all commutators
$[g_u,g_v]$ with $g_u \in G_u$, $g_v \in G_v$ and $u \adj v$. 
The right-angled Artin groups, also known as graph groups, are precisely the graph products in which
all vertex groups are infinite cyclic groups, while the right-angled Coxeter
groups arise from vertex groups of order 2.

In this article we prove three results about the graph product
$G=\G(\Gamma;G_v, v \in V(\Gamma))$  of groups $G_v$ over a graph $\Gamma$.
The first result, found in Theorem~\ref{only_if} and Theorem~\ref{SDS:Thm2},
generalises a result of Servatius, Droms and Servatius,
proved for right-angled Artin groups in \cite{SDS}; we prove a necessary
and sufficient condition on a finite graph $\Gamma$ for the kernel of the map
from $G$ to the associated direct
product to be free. 
Part of this result was already known, proved in \cite{Kim}.
The second result, Theorem~\ref{Thm:HS}, generalises
a result of Hermiller and \u{S}uni\'{c} \cite{HS} again from right-angled Artin groups;
we prove that, for a graph $\Gamma$ with finite chromatic number, $G$ has a
series in which every factor is a free product of vertex groups.
The third result, Theorem~\ref{gphyperbolic}, provides an alternative proof of a theorem due to Meier
\cite{meier}, which provides necessary and sufficient conditions on a finite
graph $\Gamma$ for $G$ to be hyperbolic.
 
While the graph product construction might seem merely to be a
generalisation of direct and free products, it has produced some groups with
rather interesting properties, including
a group (a subgroup of a right-angled Artin group)
that has FP but is not finitely presented  \cite{BestvinaBrady} and a group ($F_2 \times F_2$)  with insoluble subgroup membership
problem \cite{Mihailova}.
Right-angled Artin groups \cite{Charney07} and Coxeter groups 
have been particularly studied; their actions on CAT(0) cubical complexes
give them interesting geometrical properties.

Graph products were introduced by Green in her PhD thesis \cite{Green}, where
in particular a normal form was developed and the graph product construction
was shown to preserve  residual finiteness; this work was extended by Hsu and
Wise in \cite{HsuWise}, where right-angled Artin groups were shown
to embed in right-angled Coxeter groups and hence to be linear.
The preservation of semihyperbolicity, automaticity (as well as 
asynchronous automaticity and biautomaticity) and the possession of a 
complete rewrite system under graph products is proved in \cite{HM},
necessary and sufficient conditions for the preservation of hyperbolicity in \cite{meier}, the
question of when the group is virtually free in \cite{LohreySenizergues},
of preservation of orderability in \cite{Chiswell}. The representation of
the group as a graph product of directly indecomposable groups was shown
to be unique in \cite{Radcliffe}.
Automorphisms and the structure of centralisers for graph products
of groups have been the subject of recent study, and in particular
graph products defined over random graphs have provoked some
interest for their applications \cite{CostaFarber,CharneyRuaneStambaughVijayan,CharneyFarber}.

\section{Notation and basic properties}

We use standard notation from graph theory. So for $v,w \in V(\Gamma)$
we say that $v,w$ are adjacent if $\{v,w\}$ is an edge, and write
$v \adj w$. If $v,w$ are non-adjacent we write $v \nonadj w$.

A normal form for elements of graph products is 
developed in \cite{Green}. Each element $g \in \G(\Gamma)$ can be written as
a product $g_1 \cdots g_k$ with each $g_i$ in a vertex group
$G_{v_i}$. We shall call such a product an {\em expression} for $g$,
and the elements $g_i$ the {\em syllables} of the expression.
The expression is {\em reduced} if the syllable length $k$ is minimal among
all such expressions for $g$.  Note that the identity element of $G$ has
syllable length 0.

The operation of
replacing $g_ig_{i+1}$ by $g_{i+1}g_i$ within an expression,
when $v_i \adj v_{i+1}$, is called a {\em shuffle}.
We need the following technical result. 
\begin{proposition}
\label{reduced_exp}
\cite[Theorem 3.9]{Green}
For $g \in \G(\Gamma)$,
any expression for $g$ can be
transformed to any reduced expression for $g$ 
by a succession of shuffles and group multiplications
within the vertex groups.
\end{proposition}
Now suppose that, for each vertex $v$ of $\Gamma$, $X_v$ is a generating set
for $G_v$, and that $X = \cup_{v \in V(\Gamma)} X_v.$
The next result is a consequence of Proposition~\ref{reduced_exp}.
\begin{proposition}
\label{geodesic_word}
For $g \in \G(\Gamma)$ a non-geodesic word for $g$ over $X$ can
be reduced to a shorter representative by a succession of
shuffles followed by a replacement of a subword over $X_v$
by a shorter such word, for some vertex $v$ of $\Gamma$.
\end{proposition}
\begin{proof} 
Suppose that $w$ is a non-geodesic word over $X$ representing $g$. We can
write $w$ as a
concatenation $w_1\cdots w_k$ of maximal subwords over the generating sets
$X_v$ and hence find an expression $g_1\cdots g_k$ for $g$ for which $g_i$
is an element of a vertex group $G_{v_i}$, represented by the subword $w_i$.

We prove the result by induction on $k$.

If the expression is reduced, then one of the subwords $w_i$ must
be non-geodesic, and hence can be replaced by a shorter subword.
This includes the case where $k=1$.

Otherwise, the expression can be transformed to a reduced expression by
a succession of shuffles and multiplications within
vertex groups. The first such multiplication must be applied to an
expression represented by a word that can be derived from $w$ by a series of
shuffles only. The expression resulting from that multiplication,
which is represented by the same word, has shorter syllable length,
and the result now follows by induction.
\end{proof}

\section{Generalising Servatius, Droms and Servatius}
In this section we generalise two theorems of Servatius, Droms and Servatius.

In \cite[Theorem 1]{SDS} the following is proved:
\begin{quote} Let $G$ be the right-angled Artin group
defined by the $n$-cycle $C_n$ with $n \ge 3$.
Then $G'$ has a subgroup isomorphic to the 
fundamental group of the orientable surface of genus $1+(n-4)2^{n-3}$.
\end{quote}
We generalise this to prove the same statement
for a graph product of groups, namely:
\begin{theorem}
\label{only_if}
Let $G=\G(\Gamma;G_v, v \in V(\Gamma))$ be a graph of non-trivial groups,
for which $\Gamma$ is 
a cycle of length $n$ with $n \ge 3$.
Then the kernel of the natural map from $G$ to $G_{v_1} \times \cdots \times
G_{v_n}$ contains the fundamental group of the orientable
surface of genus $1 + (n-4)2^{n-3}$.
\end{theorem}
In fact this result was already known for $n \geq 5$ (and it is straightforward
to prove for n=3,4); however our proof is short and elementary.
For $n \geq 5$ the result is proved in Sang-hyun Kim's 2007 thesis
\cite[Theorem 3.6]{Kim} using Van Kampen diagrams, 
and for $n\geq 6$ another quite different proof is provided in
\cite[Corollary 1.3]{FuterThomas}.

In \cite[Theorem 2]{SDS} the following is proved:
\begin{quote}
Let $G$ be the right-angled Artin group defined by a finite graph
$\Gamma$. Then $G'$
is free if and only if $\Gamma$ does not contain $C_n$ as an
induced subgraph for any $n \geq 4$.
\end{quote}
We generalise this in the following result.

\begin{theorem}
\label{SDS:Thm2}
Let $G=\G(\Gamma;G_v, v \in V(\Gamma))$ be a graph product of non-trivial
groups with $V(\Gamma) = \{v_1,\ldots,v_n\}$ finite.
Then the kernel of the natural map from $G$ to $G_{v_1} \times \cdots \times
G_{v_n}$ is free if and only if $\Gamma$ does not contains a cycle of length
$4$ or more as an induced subgraph.
\end{theorem}

From now on we shall use the notation $\K(G)$ for the kernel of that
natural map that appears in both theorems.

\begin{proofof}{Theorem~\ref{only_if}}
This proof closely follows and generalises the proof of 
\cite[Theorem 1]{SDS}. We have modified the notation and terminology in places.

We begin with a technical lemma.
\begin{lemma}
\label{H_lemma}
Suppose that $\Gamma$ is a graph, $\Delta$ 
an induced subgraph of $\Gamma$,
and that $G_v, v \in V(\Gamma)$, $H_v, v \in V(\Delta)$
are groups with $H_v \subseteq G_v$ for all $v \in V(\Delta)$.
Suppose that
$G=\G(\Gamma; G_v (v \in V(\Gamma))$ and $H=\G(\Delta;G_v (v \in V(\Delta))$.
Then 
\begin{mylist}
\item[(1)] $H$ embeds naturally in $G$
as the subgroup $\langle H_v: v \in V(\Delta) \rangle $, and
\item[(2)] $\K(H)=\K(G) \cap H$.
\end{mylist}
\end{lemma}
\begin{proof}
The natural embedding of $H$ in $G$ as a subgroup  follows from
\cite[Proposition 3.31]{Green}, and the rest is then immediate.
\end{proof}

Applying Lemma~\ref{H_lemma} in the case where $\Delta=\Gamma$ and
the $H_v$'s are cyclic subgroups of the $G_v$'s, we see that 
it is sufficient to prove the result of the theorem
in the case where the vertex groups are non-trivial cyclic.
In that case, the kernel is simply $G'$.

We choose generators $x_v$ ($v \in V(\Gamma)$) for the cyclic groups $G_v$
and suppose that $x_v$ has order $k_v \in \N \cup \{\infty\}$, with $k_v>1$.

The group $G$ has the presentation
$$\langle x_v,\;(v \in V(\Gamma))\mid x_v^{k_v},[x_v,x_u]\;(v \in
V(\Gamma),\,u \adj v) \rangle .$$
We define $X_G$ to be the associated presentation complex of $G$, which has a
single 0-cell,
a 1-cell for each $x_i$ and a 2-cell for each relator, attached along
the appropriate sequence of 1-cells. Each relator is either a
commutator of generators or a power of a generator, and we call the two kinds of
associated 2-cells {\em commutator} cells and {\em power} cells.

We write $K$ for $\K(G)=G'$
and define $X_{G/K}$ to be the presentation complex of $G/K$, using the
presentation
$$\langle x_v\;(v \in V(\Gamma)) \mid x_v^{k_v},[x_v,x_u]\;(v,u
\in V(\Gamma)) \rangle .$$
Clearly $X_G$ embeds in $X_{G/K}$.

We define $\tilde{X}_{G/K}$ to be the universal cover of $X_{G/K}$,
and then $Z$ to be the subcomplex of $\tilde{X}_{G/K}$ obtained from
it by deleting the lifts of all 2-cells of $X_{G/K}$ that correspond
to commutators not present in the presentation of $G$.  
Then $Z$ covers $X_G$ and has $K$ as its fundamental group.

Now let $w$ be a word in the generators $x_v$ representing the identity in $G$.
Then it is a consequence of Proposition~\ref{reduced_exp} (using an argument
similar to the proof of Proposition~\ref{geodesic_word}), that $w$ can be
transformed to the empty word $\epsilon$ by a combination of the
following moves:
\begin{mylist}
\item[M1] delete a subword $x_vx_v^{-1}$ or $x_v^{-1}x_v$,
\item[M2] delete a subword of the form $x_v^{k_v}$ or $x_v^{-k_v}$, 
\item[M3] replace a subword $x_v^{\pm 1}x_u^{\pm 1}$ by
$x_u^{\pm 1}x_v^{\pm 1}$, where $u \adj v$.
\end{mylist}

Since the 1-skeleton of $\tilde{X}_{G/K}$ can be identified with the Cayley graph of
$G/K$, and is also equal to the 1-skeleton of $Z$, 
the edges in the 1-skeleton of $Z$ can be oriented and
labelled by the generators $x_v$ of $G$.
Suppose that $p$ is a path in $Z$ that is homotopic to the trivial loop,
and let $w$ be the label of $p$.  Then $w$ represents the identity of $G$,
and there is a sequence of words
$w=w_1,w_2,\ldots,w_k=\epsilon$, such that each
$w_i$ transformed into $w_{i+1}$ by a move of type M1, M2 or M3,
and a corresponding sequence of loops 
\[ p=p_1\rightarrow p_2\rightarrow \cdots \rightarrow p_k,\]
with $p_k$ the trivial loop, where $p_i$ is labelled by $w_i$.
Then $p_{i+1}$ is contained within $p_i$ if the move has type M1 or M2,
or within $p_i \cup F$,
where $F$ is a face of $Z$ intersecting $p_i$ in 2 consecutive edges,
if the move has type M3.

So now if $Y$ is a subcomplex of $Z$ such that
\begin{mylist}
\item[(1)] $Y$ contains any power cell of $Z$ of which it contains
the 1-skeleton
and
\item[(2)] $Y$ contains any commutator cell of $Z$ (and its incident
vertices and edges) of which it contains two adjacent edges,
\end{mylist}
then the natural embedding of $Y$ in $Z$ induces an embedding of
$\pi_1(Y)$ as a subgroup of $\pi_1(Z)$.

Both conditions are satisfied by the following subcomplex $Y$.
Fix a vertex $v$ of $Z$. Then $Y$ consists of 
all the vertices and edges of $Z$ on
paths from $v$ labelled by words $x_{v_1}\cdots x_{v_k}$ with all $v_j$ distinct,
together with all 2-cells of $Z$ whose 1-skeletons are in $Y$. Intuitively
$Y$ is the intersection of the unit cube with $Z$; note that $Y$ contains no
power cells, and so this is the same subcomplex as in \cite{SDS}; hence
Condition (1) (which does not arise in \cite{SDS}) holds vacuously.

Just as in \cite{SDS} we observe that $Y$ is an orientable
surface of Euler characteristic $(4-n)2^{n-2}$ and hence genus $1-\frac{1}{2}
\chi(Y)= 1+(n-4)2^{n-3}$.
\end{proofof}


The remainder of this section is devoted to the proof of
Theorem~\ref{SDS:Thm2}.
The `only if' part of the theorem follows from Theorem
~\ref{only_if} and Lemma~\ref{H_lemma}.
Then the `if' part is an immediate consequence of the following result.
\begin{proposition}
\label{if}
Let $\Gamma$ be a finite graph that does not contain a cycle of length $4$  or
more as an induced subgraph, let $G_v$, for $v \in V(\Gamma)$,
be arbitrary groups, and let $G=\G(\Gamma;G_v, v \in V(\Gamma))$.
Then $\K(G)$ is free.
\end{proposition}

The proof of Proposition~\ref{if} follows closely and generalises the proof of 
\cite[Theorem 2]{SDS}.

\begin{proofof}{Proposition~\ref{if}}
Write $K=\K(G)$.
If $\Gamma$ is the complete graph, then $G$ is the direct
product of the vertex groups and $K$ is the identity subgroup,
so the result holds.

The proof is by induction on the number $n$ of vertices, the case $n=1$ being covered by the above.

For $n>1$ and $\Gamma$ not complete, it follows from
\cite[Solution to Problem 9.29b]{Lovasz} (as explained in the proof of the lemma in \cite{Droms}) that $\Gamma$ can be
written as the union of proper subgraphs $\Gamma_1$ and $\Gamma_2$,
whose intersection $\Gamma_{12}$ is either complete,
or empty (if $\Gamma$ is disconnected).
Then $G\cong G_1 *_{G_{12}} G_2$,
where $G_1 = \G(\Gamma_1;G_v, v \in V(\Gamma_1))$,
$G_2 = \G(\Gamma_2;G_v, v \in V(\Gamma_2))$, and
$G_{12} = \G(\Gamma_{12};G_v, v \in V(\Gamma_{12}))$.
Now by Lemma~\ref{H_lemma}, $\K(G_1)=K \cap G_1$, $\K(G_2)=K \cap G_2$
and $\K(G_{12})=K \cap G_{12}$.
By the induction hypothesis, $\K(G_1)$ and $\K(G_2)$ are free and,
since $\Gamma_{12}$ is complete,
$\K(G_{12})$ is the identity subgroup, and hence so is $K \cap G_{12}$.
Now, as an amalgamated free product, $G$ is the fundamental group of a
graph of groups (as defined in \cite{DD}), for the graph consisting of
a single edge, with vertex groups $G_1$, $G_2$ and edge group
$G_{12}$. Since $K \cap G_{12}$ and hence also (by the normality of $K$)
the intersections of $K$ with all conjugates of $G_{12}$ are trivial,
we can apply \cite[Theorem I.7.7]{DD}
to deduce that $K$ is a free product of a free group $F$ with
subgroups of conjugates of the vertex groups $\K(G_1)$ and $\K(G_2)$. So $K$ is free
as claimed.
\end{proofof}

\section{Generalising Hermiller and  \u{S}uni\'{c}}

\begin{theorem}\label{Thm:HS}
Suppose that $\Gamma$ is a graph with finite chromatic number $\ch(\Gamma)$, and
$G=\G(\Gamma; G_v, v \in V(\Gamma))$ a graph of groups.
Then there is a series
\[ G = G_0 \rhd G_1 \rhd \cdots \rhd G_{\ch(\Gamma)} = 1,\] 
such that each factor $G_{i-1}/G_i$ is the free product of copies of 
vertex groups $G_v$.
\end{theorem}
\begin{proof}
The proof is based on the proof of \cite[Theorem A]{HS}, 
which deals with the special case where $G$ is a right-angled Artin group.
We use induction
on $\ch(\Gamma)$. If $\ch(\Gamma)=1$ then there are no edges, and so $G$ is
a free product of the vertex groups $G_v$.

Now suppose that $\ch(\Gamma)>1$. Then given a colouring of the graph with
$\ch(\Gamma)$ colours, we define the subset $D$ of the vertices to be
all those of one particular colour, and put $L=V(\Gamma)\setminus D$.
(This follows the notation of \cite{HS}, where $D$ and $L$ are
referred to as the {\em dead} and {\em living} vertices.)
Let $\Gamma_L$ be the induced subgraph with vertex set $L$,
and define $G_L:=\G(\Gamma_L;G_v, v \in L)$.
For each vertex $v$, we let $\langle X_v \mid S_v\rangle$
be a presentation for $G_v$; and we let $\langle X_L \mid S_L \rangle$ be a
presentation for $G_L$, where $X_L = \cup_{l \in L}X_l$ and $S_L$ is the
union of $S_l\ (l \in L)$ and the commutator relations derived from $\Gamma_L$.

The result holds for $G_L$ by the induction hypothesis.
Hence we can prove the result by showing that $G$ is isomorphic
to a split extension of the form $H \rtimes G_L$, where $H$ is a free 
product of vertex groups.

We define $H:= *_{d \in D} H_d $ to be the free product of groups $H_d$,
for $d \in D$, where the $H_d$ are defined as follows.
For each $d \in D$, we define $T_d$ to be the set of all 
those elements $t \in G_L$ for which there is no
reduced expression starting with an element of $G_v$ with $v  \adj d$;
note that $1 \in T_d$.
Then we define $H_d := *_{t \in T_d} G_{d,t} $ to be a free product of copies
$G_{d,t}$ of $G_d$, which are indexed by the elements of $T_d$.

For each $d \in D$, we shall now define an action of $G_L$ on $T_d$. This
induces an action of $G_L$ on $H_d$ by permuting the free factors, and this
in turn induces an action of $G_L$ on $H$, which will be used to define
the semidirect product $H \rtimes G_L$.

We start by defining an action of each $G_l$ on $T_d$.
For each element $a$ of some $G_l$ ($l \in L$),
we define a map $\alpha_{d,a}: T_d \rightarrow T_d$, by:
\[ t\alpha_{d,a}:=\left\{ \begin{array}{rl}
               ta, &  ta \in T_d \\
               t, &  ta \not \in T_d\end{array} \right.\]
Using (without modification) the argument in the proof of \cite[Theorem A]{HS},
we find that
\begin{quote}
$(*)$ if $t \in T_d$ and $1 \neq a \in G_l$, then $ta \not \in T_d$ if and only
if $d\adj l$
and all the syllables in any reduced expression for $t$ are from vertex
groups $G_v$ with $v \adj l$.
\end{quote}
Hence if $a,b$ are non-identity elements of the same vertex group $G_l$, and
$t \in T_d$, then either $t$ is fixed by both $\alpha_{d,a}$ and $\alpha_{d,b}$ or $t$ is
fixed by neither of $\alpha_{d,a}$ and $\alpha_{d,b}$.
It follows from this observation that for any $a,b \in G_l$ (including the
identity) $\alpha_{d,a}\alpha_{d,b} = \alpha_{d,ab}$.
So for each vertex group $G_l$ with $l \in L$, $a \mapsto \alpha_{d,a}$
defines an action of $G_l$ on $T_d$.

Exactly as in the proof of \cite[Theorem A]{HS}
we that, if $a\in G_l$ and $b \in G_m$ with $l \adj m$,
then $\alpha_{d,a}\alpha_{d,b}=\alpha_{d,b}\alpha_{d,a}$.

Hence we can extend our actions of the vertex groups of $L$ on $T_d$ to 
an action of $G_L$ on $T_d$ by
defining $\alpha_{d,a_1\ldots a_k}$ to be the composite $\alpha_{d,a_1}\cdots
\alpha_{d,a_k}$, where each $a_i$ is in some vertex group.

We also observe that
\begin{quote}
$(**)$ if $t$ is in the subset $T_d \subseteq G_L$ and
$\tau$ is a reduced expression
for $t$, then every prefix of $\tau$ is also a reduced expression
for an element of $T_d$, and hence $1\alpha_{d,t} = t.$
\end{quote}

Now for $t \neq 1$ we let $\theta_{d,t}$ be an isomorphism from
$G_{d,1}$ to $G_{d,t}$, and we identify $G_{d,1}$ with $G_d$.
Then, for each $x \in G_L$,
we define an isomorphism $\beta_x$ of $H$ that permutes the free factors
$G_{d,t}$ of each of the free factors $H_d$ of $H$
by defining, for  $g \in G_{d,t}$: 
\[ g\beta_x:= g\theta_{d,t}^{-1}\theta_{d,t\alpha_{d,x}}.\]
Hence $x \mapsto \beta_x$ defines an action of $G_L$ on $H$,
and we define the associated semidirect product $H \rtimes G_L$.


Our proof is complete once we have shown that $H \rtimes G_L$
is isomorphic to $G$.

For each $d \in D$, we define $X_{d,1} := X_d$ to be our given
generating set of $G_{d,1} = G_d$ and, for each $1 \neq t \in T_d$, let
$X_{d,t} = \{ x \theta_{d,t} \mid \ x \in X_d \}$ be the associated
generating set of $G_{d,t}$. Then
\begin{eqnarray*}
 H \rtimes G_L &=& \langle X_L \cup \bigcup_{d \in D}\bigcup_{t \in T_d}X_{d,t}
\mid S_L \cup \bigcup_{d \in D}( S_d \cup R_d) \rangle\\
\hbox{\rm where}\quad R_d &=& \bigcup_{t \in T_d} \{ g^x = g\beta_x:
g \in X_{d,t},x \in X_L\}.\end{eqnarray*}
For each $t$ in $\bigcup_{d \in D}T_d$,
choose a word $\rho(t)$ over $X_L$ that represents $t$.
We deduce from our observation $(**)$ above that, for $g \in G_{d,1}$ and
$t \in T_d$, we have $g\beta_t = g \theta_{d,t}$ and hence
(by the definition of the semidirect product using the action defined by
$\beta_t$) $g^{\rho(t)} = g\theta_{d,t}$.
We use these expressions to eliminate from the presentation all the generators
in the sets $X_{d,t}$, for each $d\in D$ and $1 \neq t \in T_d$.
Let $g^x=g\beta_x$ be a relator from $R_d$ with $g \in X_{d,t}$.
Then, if $t \neq 1$, $g$ rewrites as $g_1^{\rho(t)}$,
where $g_1=g\theta_{d,t}^{-1} \in X_d$, and we put $g_1=g$ if $t=1$.
So $g^x$ rewrites (or remains when $t=1$) as $g_1^{\rho(t)x}$,
while $g\beta_x$ becomes $g_1^{\rho(tx)}$ when $tx \in T_d$ or
$g_1^{\rho(t)}$ when $tx \not\in T_d$.  In the first case,
$\rho(tx) =_{G_L} \rho(t)x$, and the relation is a consequence of relations
of $G_L$, and so can be omitted.

So our presentation simplifies to
\begin{eqnarray*}
 H \rtimes G_L &=& \langle X_L \cup \bigcup_{d \in D}X_d
     \mid S_L \cup \bigcup_{d \in D}( S_d \cup R'_d) \rangle\\
\hbox{\rm where}\quad R'_d &=& 
 \{ g^{\rho(t)x} = g^{\rho(t)} : g \in X_d,t \in T_d,x \in X_L, tx \not\in T_d \}.\end{eqnarray*}
Now we know from $(*)$ that whenever $g,x,t$ satisfy the conditions in the
definition of $R'_d$, then $x$ and $t$ must commute, and $x \in X_l$ with
$l \adj d$.  Since $x,t \in G_L$,
$xt=tx$ must be a consequence of the relations in $S_L$, and hence the
corresponding relation $g^{\rho(t)x}=g^{\rho(t)}$ can be replaced by $g^x=g$.
So we can replace $R'_d$ by
\[ R''_d = 
 \{ g^x = g : g \in X_d, x \in X_l, l \adj d \}.\]
that is by the set of commutator relations implied by those edges
of $\Gamma$ between vertices in $L$ and vertices in $D$,
and we see that
 \[ H \rtimes G_L = \langle X_L \cup \bigcup_{d \in D}X_d \mid S_L \cup \bigcup_{d \in D}( S_d
\cup R''_d) \rangle,\] which
we recognise as a presentation for $\G(\Gamma)$.
\end{proof}

The following immediate consequence is already known \cite{Chiswell}.
\begin{corollary}
If $G_v, v \in V(\Gamma)$ are right orderable
groups, then for any graph $\Gamma$ with finite chromatic number,
the graph group $\G(\Gamma; G_v, v \in V(\Gamma))$ is right orderable.
\end{corollary}
\begin{proof} Let $G=\G(\Gamma)$.
We find a series 
\[ G = G_0 \rhd G_1 \rhd \cdots \rhd G_{\ch(\Gamma)} = 1\] 
for $G$, where each factor $G_{i-1}/G_i$ is a free product of groups $G_v$.
The result now follows from the fact that the set of right orderable groups is
closed under free products and extensions.
\end{proof}

\section{When is a graph product of groups hyperbolic?}
In this section, we provide a alternative proof of the following
theorem, which is the main result of~\cite{meier}. Our proof makes use of
the result proved in~\cite{papasoglu} that a group is (word-)hyperbolic if
and only if geodesic bigons in its Cayley graph are uniformly thin, and
is otherwise elementary.

\begin{theorem}
\label{gphyperbolic}
Let $G=\G(\Gamma;G_v, v \in V(\Gamma))$ be a graph product of non-trivial
groups with $V(\Gamma) = \{v_1,\ldots,v_n\}$ finite.
Then $G$ is hyperbolic if and only if all the following conditions hold.
\begin{mylist}
\item[(i)] Each vertex group $G_v$ is hyperbolic.
\item[(ii)] No two vertices of $\Gamma$ with infinite vertex groups are adjacent.
\item[(iii)] If $G_v$ is infinite for some vertex $v$, then any two vertices
that are adjacent to $v$ are adjacent to each other.
\item[(iv)] $\Gamma$ has no cycle of length 4 as an induced subgraph.
\end{mylist}
\end{theorem}

\begin{proof}
For each vertex $v$ of $\Gamma$ we choose a finite symmetric generating set
$X_v$ of $G_v$, equal to $G_v \setminus \{1\}$ if $G_v$ is finite.
Then we let $X=\cup_{v \in V(\Gamma)} X_v$.
We denote by $v(x)$ the vertex of $\Gamma$ such that $x \in X_{v(x)}$.

Suppose first that $G$ is hyperbolic.
Then geodesic triangles in the Cayley graph $\C{G}{X}$ are uniformly thin.
It follows from
Proposition~\ref{geodesic_word} that, for any $v \in X_v$, words labelling
geodesics within $\C{G_v}{X_v}$ remain
geodesic within $\C{G}{X}$ and so geodesic triangles within
$\C{G_v}{X_v}$ are also uniformly thin. Hence each $G_v$ is hyperbolic,
that is Condition (i) holds. 
Just as in the proof in~\cite{meier}, we see easily
that if any of the Conditions (ii)--(iv) fails, then $G$ contains a subgroup
isomorphic to $\Z^2$, and so cannot be hyperbolic.

This completes the `only if' part of the result. The `if' part is a consequence 
of the following result together with the main result of \cite{papasoglu}.
\end{proof}

\begin{proposition}
\label{prop:gphyp}
Let $G = \G(\Gamma,G_v, v \in V(\Gamma))$ be a graph product
as in Theorem~\ref{gphyperbolic} and such that
the Conditions (i)--(iv) are satisfied.
Then geodesic bigons in $\C{G}{X}$ are uniformly thin.
\end{proposition}

Before embarking on the proof of this proposition, we prove two technical
lemmas, which will be used repeatedly.

We first clarify our notation.
We let $n=|V(\Gamma)|$.
If $\alpha$ is a word over a generating set $X$ then, for $x \in X$,
$x \in \alpha$ will mean that $x$ or $x^{-1}$ occurs in the word $\alpha$.

\begin{lemma}
\label{lem:shuffle_ft}
Let $\alpha$ be a word over $X$. Suppose that a sequence of shuffles
of $|\alpha_2|$ letters to the right hand end transforms $\alpha$
to $\alpha_1\alpha_2$. Then $\alpha$ and $\alpha_1\alpha_2$ fellow travel at
distance at most $2 \min(|\alpha_1|,|\alpha_2|).$
\end{lemma}
\begin{proof}
We prove by induction on $|\alpha_2|$ that $\alpha$ and  $\alpha_1\alpha_2$
fellow travel at distance at most $2|\alpha_2|$. If $|\alpha_2|=1$ then there
are words $\alpha_{11},\alpha_{12}$ with
$\alpha = \alpha_{11}\alpha_2\alpha_{12}$ and
$\alpha_1 = \alpha_{11}\alpha_{12}$ such that $\alpha_2$ commutes with
all letters in $\alpha_{12}$. It is then easily verified that
$\alpha= \alpha_{11}\alpha_2\alpha_{12}$ and
$\alpha_1\alpha_2 = \alpha_{11}\alpha_{12}\alpha_2$ 2-fellow travel.
For $|\alpha_2|>1$, first shuffle the last letter of $\alpha_2$ to the end of
the word. Then the result follows from the case $|\alpha_2|=1$ and the
inductive hypothesis. By a similar argument, $\alpha$ and  $\alpha_1\alpha_2$
fellow travel at distance at most $2|\alpha_1|$, and the result follows.
\end{proof}

\begin{lemma}
\label{lem:nongeo} 
For $x \in X_v$, suppose that  $\alpha$ is a geodesic word over $X$,
but $\alpha x$ is non-geodesic.
Then $\alpha$ has a suffix $\beta$, with first letter in $X_v$, 
and all other letters either in $X_v$ or in $X_w$ with $w \adj v$,
such that $\beta x$ is non-geodesic.
If $G_v$ is finite, then we can choose $\beta$ so that only its first
letter is in $X_v$.
\end{lemma}
\begin{proof}
We choose $\beta$ to be the shortest suffix of $\alpha$ such that $\beta x$ is not geodesic,
and let $\beta=z\gamma$ with $z \in X$. 
Now we apply Proposition~\ref{geodesic_word} to see that
we can replace $\beta x$ by a
$G$-equivalent shorter subword by first shuffling generators from commuting 
vertex groups, and then replacing a word from the generators in one of the
vertex groups $G_u$ by a $G_u$-equivalent shorter word. If $u \ne v$, then
we could carry out those shuffles that do not involve $x$ in $\beta$ and shorten
$\beta$, contradicting the geodesity of $\beta$. So we must have $u = v$,
and similarly $u=v(z)$, so $v(z)=v$ as claimed. Furthermore, the word
over $X_u$ that is shortened must start with $z$ and end with $x$.

Let $y \in \beta $ with $v(y) \ne v$. If $v(y) \nonadj v$,
then $y$ would remain between $z$ and $x$
after the shuffles,
and so it would not be possible to create a subword over $X_v$ starting with
$z$ and ending with $x$; hence  we must have $v(y) \adj v$.

The final statement
now follows immediately from the facts that the choice of generating
set for a finite vertex group ensures that geodesic words in that group
have length at most 1.
\end{proof}


The remainder of the section is devoted to the proof of
Proposition~\ref{prop:gphyp}.

We need to prove the uniform thinness of geodesic bigons.
From the Condition (i) that vertex groups are hyperbolic we deduce
the existence of a constant $k$ such that for each vertex $v$
bigons over $X_v$ are uniformly $k$-thin.

The endpoints of
geodesic bigons in $\C{G}{X}$ may be at vertices or on edges of
$\C{G}{X}$ but it is sufficient to prove their uniform thinness
in two situations: when both endpoints of the bigon lie on vertices of
$\C{G}{X}$, and when one endpoint lies at a vertex and the other lies at the
midpoint of an edge. In other words, we have to prove that there exist
constants $K,K'$ such that:
\begin{mylist}
\item[(a)] If $\alpha_1,\alpha_2$ are geodesic words over $X$ with $\alpha_1=_G \alpha_2$ than
$\alpha_1$ and $\alpha_2$ $K$-fellow travel;
\item[(b)] If $\alpha_1,\alpha_2$ are geodesic words over $X$ with $|\alpha_1|=|\alpha_2|$
and $\alpha_1x=_G \alpha_2$ for some $x \in X$, then $\alpha_1$ and $\alpha_2$ $K'$-fellow travel.
\end{mylist}

In fact we can reduce consideration of Case (b) easily to Case (a),
as follows.

Suppose that $\alpha_1,\alpha_2,x$ are as in (b), and
let $\alpha_1 = \eta\beta$, where $\beta$ is the
suffix of $\alpha_1$ arising from Lemma~\ref{lem:nongeo}.
Let $v=v(x)$.
Now suppose that shuffling all the letters of $X_v$ to the right
hand end transforms $\beta$ to $\beta'_v\beta_v$, where $\beta$ is a maximal
suffix over $X_v$.
It follows from Lemma~\ref{lem:nongeo} that
$\beta_v x$ is non-geodesic, and hence equal in $G_v$ to some word $\gamma$ over $X_v$ with
$|\gamma| = |\beta_v|$.   

If $|G_v|$ is finite, then $|\beta_v|=|\gamma|=1$, and this fact together
with Lemma~\ref{lem:shuffle_ft}
ensures that $\beta x$ and $\beta'_v\gamma$ 2-fellow
travel.

If $|G_v|$ is infinite then the Conditions (ii) and (iii) imply that
the vertex groups $G_u$ containing the letters of $\beta'_v$
are finite and commute with each other. Then geodesity of $\beta'_v$
implies that for any vertex $u$ at most one letter of $\beta'_v$ comes from
$X_u$, and so
$|\beta'_v| \le n$, and we can apply Lemma~\ref{lem:shuffle_ft}
to see that $\beta x$ $2n$-fellow travels
with  $\beta'_v\beta_vx$. Since $\beta_v x$ and $\gamma$ $k$-fellow travel,
it follows that $\beta x$ and $\beta'_v\gamma$ $(2n+k)$-fellow travel.

In either case, we see that $\alpha_1x=\eta\beta x$ and $\eta\beta'_v\gamma$
$(2n+k)$-fellow travel and so, since $\eta\beta'_v\gamma$ is
a geodesic equal in $G$ to $\alpha_2$,  (b) would follow from (a). 

So now assume that $\alpha_1$ and $\alpha_2$ are as in (a).
We shall prove by induction on $|\alpha_1|$ that they fellow travel at
distance $K=4n + k$. This is clear when $|\alpha_1|=0$, so assume
that $|\alpha_1|>0$.
where 
Let $\alpha_1 = \eta_1\gamma_1$ and $\alpha_2=\eta_2\gamma_2$, where,
for $i=1,2$, $\gamma_i$ is the
maximal suffix of $\alpha_i$ all of whose letters come from commuting
vertex groups.
Geodesity of $\alpha_i$ implies that $\gamma_i$ can
contain at most one letter from each finite vertex group.

Suppose first that at least one of the words $\gamma_i$, say $\gamma_1$,
contains a letter from an infinite vertex group, $G_v$.

Let $y$ be the last such letter in $\gamma_1$.
Then $y$ can be shuffled to the end of $\gamma_1$ and so $\alpha_2 y^{-1}$ is
not geodesic. We apply Lemma~\ref{lem:nongeo} to $\alpha_2$ and $y^{-1}$, and let $\beta$ be the
resulting suffix of $\alpha_2$. By Conditions (ii) and (iii), for each
$z \in \beta$, either $v(z)=v$ or $G_{v(z)}$ is finite, and any two such
vertices are adjacent in $\Gamma$.
So $\beta$ is a suffix of $\gamma_2$, and hence $\gamma_2$ also
contains a letter from $X_v$. From Condition (ii), we see that
$G_v$ is the unique infinite vertex group with letters in
$\gamma_1$ or $\gamma_2$.

For $i=1,2$, let $\gamma_{i,v}$ and $\gamma'_{i,v}$ be the words obtained from $\gamma_i$ by
deleting all letters not in $X_v$, or in $X_v$, respectively.
Conditions (ii) and (iii) ensure that $|\gamma'_{i,v}| \le n$, and hence
by Lemma~\ref{lem:shuffle_ft} 
$\gamma_i =_G \gamma'_{i,v}\gamma_{i,v}$, and $\gamma_i$ $2n$-fellow
travels with 
$\gamma'_{i,v}\gamma_{i,v}$.

We claim that $\gamma_{1,v} =_G \gamma_{2,v}$. For if not, assume without loss
that $|\gamma_{2,v}| \le |\gamma_{1,v}|$, and let $\alpha$ be a geodesic word over $X_v$ for
$\gamma_{2,v}\gamma_{1,v}^{-1}$. Then $\eta_2\gamma'_{2,v}\alpha =_G \eta_1\gamma'_{1,v}$ with
$|\eta_2\gamma'_{2,v}\alpha| > |\eta_1\gamma'_{1,v}|$
and  we can apply Lemma~\ref{lem:nongeo} again to the shortest
non-geodesic prefix of $\eta_2\gamma'_{2,v}\alpha$ and conclude that $\gamma'_{2,v}$ contains a
letter from $X_v$, a contradiction.

Hence $\gamma_{1,v} =_G \gamma_{2,v}$,
and so (since $\alpha_1=_G \alpha_2$), we also have
$\eta_1\gamma'_{1,v}=_G \eta_2\gamma'_{2,v}$.
Since $\gamma_{1,v}$ and $\gamma_{2,v}$ are geodesic they must have the same length. We know them to be non-trivial, and hence the words $\eta_1\gamma'_{1,v}$
and $\eta_2\gamma'_{2,v}$ both have length less that $|\alpha_1|$,
and we can apply the induction hypothesis to see that they
$K$-fellow travel.

The situation is displayed in Fig.~\ref{fig1}.
We deduce that $\alpha_1$ and $\alpha_2$ $K$-fellow travel, as claimed.

\setlength{\unitlength}{0.75pt}
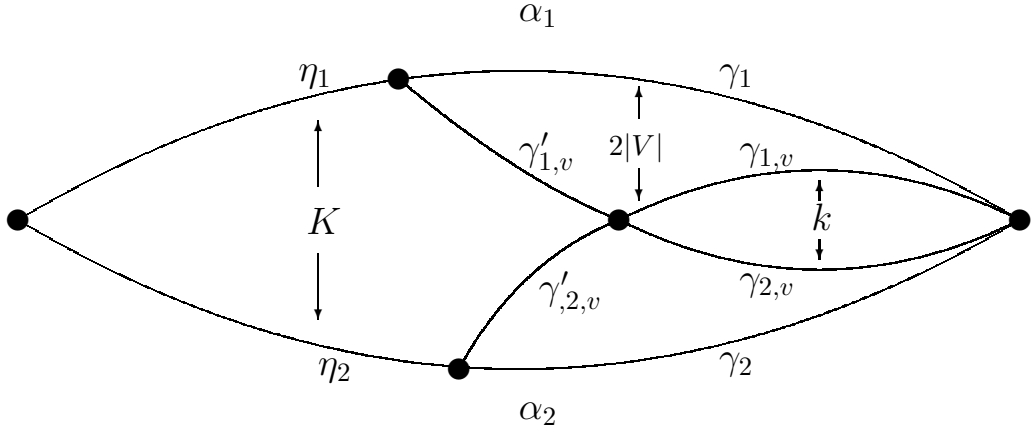
\begin{figure}
\begin{center}
{\large
\begin{picture}(500,260)(-250,-110)

\put(-250,0){\circle*{10}}
\put(-60,71){\circle*{10}}
\put(-30,-75){\circle*{10}}
\put(50,0){\circle*{10}}
\put(250,0){\circle*{10}}
\qbezier(-250,0)(0,150)(250,0) \put(-110,70){$\eta_1$}
\put(100,70){$\gamma_1$}
\put(0,100){$\alpha_1$}
\qbezier(-250,0)(0,-150)(250,0) \put(-100,-78){$\eta_2$}
\put(100,-75){$\gamma_2$}
\put(0,-100){$\alpha_2$}
\qbezier(-60,71)(0,20)(50,0) \put(0,30){$\gamma'_{1,v}$}
\qbezier(-30,-75)(0,-20)(50,0) \put(10,-40){$\gamma'_{,2,v}$}
\qbezier(50,0)(150,50)(250,0) \put(110,30){$\gamma_{1,v}$}
\qbezier(50,0)(150,-50)(250,0) \put(110,-35){$\gamma_{2,v}$}
\put(-100,17){\vector(0,1){33}}
\put(-106,-7){$K$}
\put(-100,-17){\vector(0,-1){33}}
\put(150,10){\vector(0,1){10}}
\put(146,-5){$k$}
\put(150,-10){\vector(0,-1){10}}
\put(60,51){\vector(0,1){16}}
\put(45,33){\small $2|V|$}
\put(60,26){\vector(0,-1){16}}
\end{picture}
}
\caption{\label{fig1} Infinite vertex group in tails}
\end{center}
\end{figure}

So now we may suppose that
$\gamma_1$ and $\gamma_2$ only involve letters from
finite vertex groups, and hence
that $|\gamma_i| \le n$ for $i=1,2$.
We consider two cases.

First suppose that some vertex group $G_v$ is involved in both
$\gamma_1$ and $\gamma_2$; that is, for some $x,y \in G_v$,
$x \in \gamma_1,y \in \gamma_2$.
Then by a similar argument to that in the previous paragraph, we can
use Lemma~\ref{lem:nongeo} to prove that  $x=y$ and, for $i=1,2$, we have
$\alpha_i =_G \eta_i \gamma'_{i,v} x$, where $\gamma'_{i,v}$ is the result of deleting
the single occurrence of $x$ from
$\gamma_i$.  By the inductive hypothesis, $\eta_1\gamma'_{1,v}$ and
$\eta_2\gamma'_{2,v}$ $K$-fellow travel, and
since $|\gamma_i| \le n$, we see that $\alpha_1$ and $\alpha_2$ must do too.
See Fig.~\ref{fig2}.

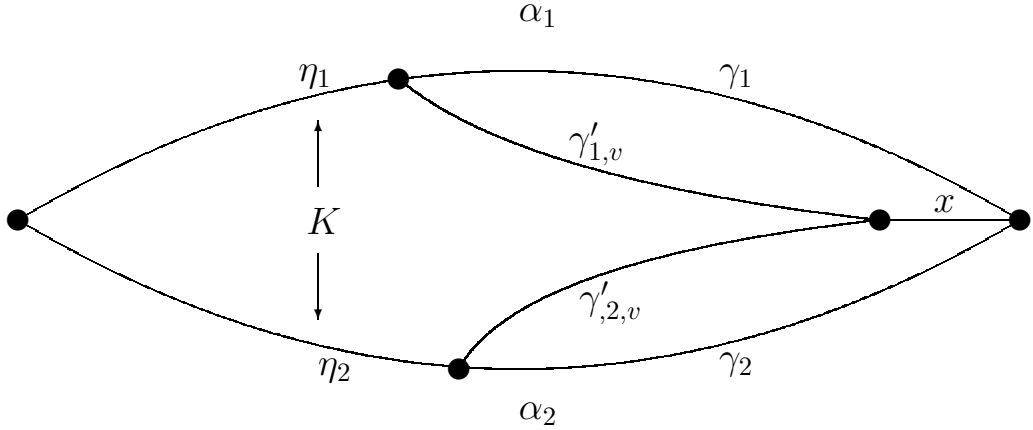
\begin{figure}
\begin{center}
{\large
\begin{picture}(500,260)(-250,-110)

\put(-250,0){\circle*{10}}
\put(-60,71){\circle*{10}}
\put(-30,-75){\circle*{10}}
\put(180,0){\circle*{10}}
\put(250,0){\circle*{10}}
\qbezier(-250,0)(0,150)(250,0) \put(-110,70){$\eta_1$}
\put(100,70){$\gamma_1$}
\put(0,100){$\alpha_1$}
\qbezier(-250,0)(0,-150)(250,0) \put(-100,-78){$\eta_2$}
\put(100,-75){$\gamma_2$}
\put(0,-100){$\alpha_2$}
\qbezier(-60,71)(0,20)(180,0) \put(25,35){$\gamma'_{1,v}$}
\qbezier(-30,-75)(0,-20)(180,0) \put(30,-45){$\gamma'_{,2,v}$}
\put(180,0){\line(1,0){70}} \put(207,4){$x$}
\put(-100,17){\vector(0,1){33}}
\put(-106,-7){$K$}
\put(-100,-17){\vector(0,-1){33}}
\end{picture}
}
\caption{\label{fig2} Common finite vertex group in tails}
\end{center}
\end{figure}

It remains to consider, and eliminate the case where no vertex group
contains letters in both $\gamma_1$ and $\gamma_2$.
If $\eta_1$ or $\eta_2$ is empty, then $|\alpha_i| \le n$
and the result follows immediately, so let $z_1,z_2$
be the final letters of $\eta_1$ and $\eta_2$, respectively.

We shall now show that $v(z_2) \adj v(x)$  for every $x \in \gamma_1$,
and $v(z_1) \adj v(y)$ for every $y \in \gamma_2$.

For suppose that  $x \in \gamma_1$. 
By definition of $\gamma_1$,$z_1 \not \in G_{v(x)}$, so $v(z_1) \neq v(x)$.
Since $x$ can be shuffled to the end of $\gamma_1$, $\alpha_2x^{-1}$
is not geodesic. We apply Lemma~\ref{lem:nongeo} to $\alpha_2$ and $x^{-1}$
and let $\beta$ be the resulting suffix of $\alpha_2$;
then $\beta$ starts with a letter from $X_{v(x)}$.
Since $v(x) \ne v(y)$ for all $y \in \gamma_2$, $\gamma_2$ must
be a proper suffix of $\beta$, and hence, by Lemma~\ref{lem:nongeo},
for each letter $y \in \gamma_2$, $v(x) \adj v(y)$.
Hence $v(x) \ne v(z_2)$, and so $z_2 \gamma_2$ is also a proper suffix of
$\beta$ and thus $v(x) \adj v(z_2)$.
Similarly, if $y \in \gamma_2$ then $v(y) \adj v(z_1)$.

Note that it follows from the above that $v(z_1) \neq v(z_2)$,
since the maximality of $\gamma_1$ as a suffix of $\alpha_1$
of letters from commuting vertex groups ensures that we cannot 
have $v(z_1) \adj v(x)$ for all $x \in \gamma_1$.

Next we show that $v(z_1) \adj v(z_2)$. For
if we multiply $\alpha_1$ by each of the letters of $\gamma_2^{-1}$ in turn,
then Lemma~\ref{lem:nongeo} ensures that each such letter
shuffles past the suffix $z_1\gamma_1$ of $\alpha_1$ and merges with
an earlier letter in $\eta_1$. We end up with a geodesic word $\alpha'_1$ with
$\alpha'_1 =_G \eta_2$, where $\alpha'_1$ still has $z_1 \gamma_1$ as a suffix. 
Then $\alpha'_1 z_2^{-1}$ is not geodesic, and Lemma~\ref{lem:nongeo} now
implies that $v(z_1) \adj v(z_2)$.

So now, we can choose $x \in \gamma_1$ with $v(z_1) \nonadj v(x)$
and $y \in \gamma_2$ with $v(z_2) \nonadj v(y)$, and then the subgraph
of $\Gamma$ induced on  $v(z_1), v(y), v(x), v(z_2)$
is a cycle of length 4, and Condition (iv) is violated.

Hence this final case cannot occur, and the proofs of
Proposition~\ref{prop:gphyp} and Theorem~\ref{gphyperbolic} are complete.

\section*{Acknowledgments}

The second author would like to thank the Mathematics
Department of the University of Neuch\^atel, Switzerland, for its generous
hospitality and support during the period of research for this article.

\end{document}